\documentclass[11pt,reqno]{amsart}
\usepackage{amssymb,latexsym}
\usepackage{fancyhdr,amssymb}
\usepackage{hyperref}
\usepackage{listings}
\usepackage{amsmath}
\usepackage[T1]{fontenc}
\usepackage[letterpaper, top = 2cm, bottom = 2cm,right = 3cm, left = 3cm]{geometry}
\theoremstyle{theorem}
\newtheorem{theorem}{Theorem}[section]
\newtheorem{lemma}[theorem]{Lemma}
\newtheorem{proposition}[theorem]{Proposition}

\theoremstyle{definition}
\newtheorem{definition}[theorem]{Definition}

\theoremstyle{remark}
\newtheorem*{remark}{Remark}

\title{Antiautomorphisms and biantiautomorphisms of some finite abelian groups}
\author{Daniel L\'{o}pez-Aguayo and Servando L\'{o}pez-Aguayo}
\address{Tecnologico de Monterrey, Escuela de Ingenier\'{i}a y Ciencias, Hidalgo, M\'{e}xico.}
\email{dlopez.aguayo@itesm.mx}
\address{Photonics and Mathematical Optics Group, Tecnologico de Monterrey, Monterrey, M\'{e}xico.} 
\email{servando@itesm.mx}

\begin{document}
\maketitle
\begin{abstract} We extend the concepts of antimorphism and antiautomorphism of the additive group of integers modulo $n$, given by Gaitanas Konstantinos in \cite{1}, to abelian groups. We give a lower bound for the number of antiautomorphisms of cyclic groups of odd order and give an exact formula for the number of linear antiautomorphisms of cyclic groups of odd order. Finally, we give a partial classification of the finite abelian groups which admit antiautomorphisms and state some open questions.
\end{abstract}

\section{Introduction}
In this paper we introduce the concept of biantiautomorphism which is a bijective antimorphism in the sense of \cite{1}. The main result of this paper is Theorem $3.8$ which gives a partial classification of the finite abelian groups which admit antiautomorphisms. The main tool for this classification is the use of generalized Wilson's theorem for finite abelian groups, the Frobenius companion matrix and the Chinese Remainder theorem. We also give an exact formula for the number of linear antiautomorphisms of cyclic groups of odd order. \\

We begin by recalling the reader Problem \textbf{2014}, proposed by Gaitanas Konstantinos in \cite{1}: \\

\textbf{Problem 2014}. For every integer $n \geq 2$, let $(\mathbf{Z}_{n},+)$ be the additive group of integers modulo $n$. Define an \emph{antimorphism} of $\mathbf{Z}_{n}$ to be any function $f: \mathbf{Z}_{n} \rightarrow \mathbf{Z}_{n}$ such that $f(x)-f(y) \neq x-y$ whenever $x,y$ are distinct elements of $\mathbf{Z}_{n}$. We say that $f$ is an antiautomorphism of $\mathbf{Z}_{n}$ if $f$ is a bijective antimorphism of $\mathbf{Z}_{n}$. For what values of $n$ does $\mathbf{Z}_{n}$ admit an antiautomorphism? \\

For the reader's convenience we first give a solution to Problem 2014.
\begin{proposition} \label{prop1} Let $n$ be an odd number. Then $\mathbf{Z}_{n}$ admits an antiautomorphism.
\end{proposition}

\begin{proof} It suffices to consider $\varphi: \mathbf{Z}_{n} \rightarrow \mathbf{Z}_{n}$ defined by $\varphi([t])=-[t]$. 
\end{proof}

\begin{proposition} \label{prop2} Let $n$ be an even number. Then $\mathbf{Z}_{n}$ does not admit an antiautomorphism.
\end{proposition}
\begin{proof} Suppose that $f$ is an antiautomorphism of $\mathbf{Z}_{2n}$, where $n \geq 2$. We claim that the set $\mathcal{A}=\{f([x])-[x]: [x] \in \mathbf{Z}_{2n}\}$ contains, at least, two repeated elements. If not, then $\mathcal{A}=\mathbf{Z}_{2n}$; and since $f$ is a bijection:
\begin{align*}
[0]&=\displaystyle \sum_{x=0}^{2n-1} \left(f[x]-[x]\right)=\left[\displaystyle \sum_{x=0}^{2n-1} x \right] \equiv [-n] \pmod{2n}
\end{align*}
The result follows.
\end{proof}

\begin{proposition} \label{prop3} Let $p$ be an odd prime. Then $\mathbf{Z}_{p}$ admits at least $p^2-2p$ antiautomorphisms.
\end{proposition}

\begin{proof} Let $\mathbf{Z}_{p}^{\ast}$ denote the group of units of $\mathbf{Z}_{p}$. Fix $[a] \in \mathbf{Z}_{p}^{\ast} \setminus \{[1]\}$ and for each $[t] \in \mathbf{Z}_{p}$, define $\varphi_{a}([t])=[at]$. It is immediate to see that $\varphi_{a}$ is an antiautomorphism. Thus we obtain at least $p-2$ antiautomorphisms. To obtain $p^2-2p$ antiautomorphisms, it suffices to consider all the translation maps of $\varphi_{a}$; that is, $\psi_{a,b}(t):=\varphi_{a}(t)+b$ where $b \in \mathbf{Z}_{p}$.
\end{proof}

\begin{remark} Note that the number of antiautomorphisms of $\mathbb{Z}_{p}$ is bounded above by the number of injective maps $\mathbb{Z}_{p} \rightarrow \mathbb{Z}_{p}$ which have a unique fixed point; this is given by $!(p-1) \cdot p$ where $!(p-1)$ denotes the subfactorial of $p-1$.
\end{remark}

We now give a lower bound for the number of antiautomorphisms of cyclic groups of odd prime power order.
\begin{proposition} \label{prop4} Let $\alpha \geq 2$ and let $p$ be an odd prime. Then $\mathbf{Z}_{p^{\alpha}}$ contains at least $p^{2\alpha}-2p^{2\alpha-1}$ antiautomorphisms.
\end{proposition}

\begin{proof} It suffices to compute the cardinality of the set $W=\{a \in [2,p^{\alpha}-1]: (a,p^{\alpha})=(a-1,p^{\alpha})=1\}$. Let $X=\{a \in [2,p^{\alpha}-1]: (a,p^{\alpha})=1\}$ and $Y=\{a \in [2,p^{\alpha}-1]: a \equiv 1 \pmod p\}$. Then $|W|=|X \setminus Y|=\varphi(p^{\alpha})-1-(p^{\alpha-1}-1)=p^{\alpha}-2p^{\alpha-1}$, where $\varphi$ denotes Euler's totient function. Therefore we obtain at least $p^{\alpha}-2p^{\alpha-1}$ antiautomorphisms, and by considering translations we get $p^{\alpha}(p^{\alpha}-2p^{\alpha-1})=p^{2\alpha}-2p^{2\alpha-1}$.
\end{proof}

\begin{remark} In Section $3$ we give an exact formula for the number of linear antiautomorphisms of cyclic groups of odd order. 
\end{remark}

\section{Antiautomorphisms}
\begin{definition}[antiautomorphism] Let $G$ be an abelian group and let $f: G \rightarrow G$ be any function. We say that $f$ is an \emph{antimorphism} if the map $id_{G}-f$ is injective. We say that an antimorphism $f$ is an \emph{antiautomorphism} of $G$ if $f$ is a bijection.
\end{definition}

\begin{remark} If $G$ is finite, then $id_{G}-f$ is bijective if and only if $id_{G}-f$ is injective/surjective.
\end{remark}

The following Proposition generalizes Proposition \ref{prop1}; moreover, it implies that any torsion-free abelian group admits an antiautomorphism.

\begin{proposition} \label{prop5} Let $G$ be an abelian group. Then the map $f: GÊ\rightarrow G$ given by $f(x)=-x$ is an antiautomorphism if and only if $G$ has no element of order $2$.
\end{proposition}

Recall that every cyclic group of even order has a unique element of order $2$; hence the following result is a natural generalization of Proposition \ref{prop2}.  

\begin{proposition} \label{prop6} Let $G$ be a finite abelian group that has exactly one element of order $2$. Then $G$ does not admit antiautomorphisms.
\end{proposition}

\begin{proof} Using generalized Wilson's theorem for finite abelian groups \cite[Theorem 2.4]{2}, we have that if $g$ is the unique element of order $2$ then $\displaystyle \sum_{h \in G} h =g$. Now suppose for the sake of contradiction that $f$ is an antiautomorphism of $G$. Since $id_{G}-f$ is a bijection, then $0=\displaystyle \sum_{h \in G} (h-f(h))=\displaystyle \sum_{h \in G} h = g$, a contradiction. It follows that $G$ does not admit antiautomorphisms.
\end{proof}

We will make use of the following lemma, whose straightforward proof we omit.

\begin{lemma} \label{lemma7} Let $G_{1},\ldots,G_{n}$ be abelian groups and let $f_{1},\ldots,f_{n}$ be antiautomorphisms of $G_{1},\ldots,G_{n}$, respectively. Then the direct sum map $\varphi=f_{1} \oplus \cdots \oplus f_{n}: \displaystyle \bigoplus_{i=1}^{n} G_{i} \rightarrow \displaystyle \bigoplus_{i=1}^{n}G_{i}$ given by $(t_{1},\ldots,t_{n}) \mapsto (f_{1}(t_{1}))\ldots, f_{n}(t_{n}))$is an antiautomorphism.
\end{lemma}
The following result gives an extension of Proposition \ref{prop3}.  \\

\begin{proposition} \label{prop8} Let $G$ be a cyclic group of order $p_{1}^{\alpha_{1}} \cdots p_{k}^{\alpha_{k}}$ where $p_{1},\ldots,p_{k}$ are distinct odd primes, and $\alpha_{i} \geq 1$ for all $i \in \{1,\ldots,k\}$. Then $G$ contains at least $\displaystyle \prod_{i=1}^{k} (p_{i}^{2\alpha_{i}}-2p_{i}^{2\alpha_{i}-1})$ antiautomorphisms.
\end{proposition}

\begin{proof} By the Chinese Remainder Theorem, $G$ is isomorphic to $\displaystyle \bigoplus_{i=1}^{k} \mathbf{Z}_{p_{i}}^{\alpha_{i}}$. By Proposition \ref{prop4}, for each $i \in \{1,\ldots,k\}$, we get $p_{i}^{2\alpha_{i}}-2p_{i}^{2\alpha_{i}-1}$ antiautomorphisms. Now applying Lemma \ref{lemma7} yields the desired result. \end{proof}

Even though $\mathbf{Z}_{2}$ does not admit antiautomorphisms, we now show that for every $r \geq 2$, the group $\mathbf{Z}_{2}^{r}$ always admits antiautomorphisms; here $\mathbf{Z}_{2}^{r}$ denotes the direct sum of $r$-copies of $\mathbf{Z}_{2}$. 
\begin{proposition} \label{prop9} For every $r \geq 2$, the group $\mathbf{Z}_{2}^{r}$ admits antiautomorphisms. 
\end{proposition}
\begin{proof} We first deal with the case $r=2$, the Klein four-group. In this case, the map $f: \mathbf{Z}_{2}^{2} \rightarrow \mathbf{Z}_{2}^{2}$ given by:
\begin{align*}
&(1,1) \mapsto (0,0) \\ 
&(0,1) \mapsto (0,1) \\
&(0,0) \mapsto (1,0) \\
&(1,0) \mapsto (1,1)
\end{align*}
is readily verified to be an antiautomorphism (in fact, there are $8$ antiautomorphisms of $\mathbf{Z}_{2}^{2}$). In the case $r=2k$ is an even number greater or equal than $4$, then by Lemma \ref{lemma7}, we get that the map $\underbrace{f \oplus f \oplus \cdots \oplus f}_{\text{$k$-times}}$ is an antiautomorphism of $\mathbf{Z}_{2}^{r}$. Therefore, it suffices to prove the result when $r$ is an odd number greater or equal than $3$. \\

We now deal with the group $\mathbf{Z}_{2}^{3}$. We are required to find a bijection $f$ of $\mathbf{Z}_{2}^{3}$ such that $id_{\mathbf{Z}_{2}^{3}}-f$ is also a bijection. Note that $\mathbf{Z}_{2}^{3}$ has $8!=40320$ bijections. With the aid of Matlab (see the Appendix) one can find all antiautomorphisms of $\mathbf{Z}_{2}^{3}$. In fact, there are $384$ such antiautomorphisms; although for our purposes we only need to find one. An explicit antiautomorphism $\varphi$ of $\mathbf{Z}_{2}^{3}$ is given by:
\begin{align*}
& (1,1,1) \mapsto (0,0,0) \\
& (1,0,1) \mapsto (0,0,1) \\
& (0,1,1) \mapsto (0,1,0) \\
& (0,0,1) \mapsto (0,1,1)  \\
& (0,1,0) \mapsto (1,0,0) \\
& (0,0,0) \mapsto (1,0,1)  \\
& (1,1,0) \mapsto (1,1,0) \\
& (1,0,0) \mapsto (1,1,1) 
\end{align*}
Note that every odd number, greater or equal than $5$, can be written in the form $2t+3$ where $t \geq 1$. Then if $r=2t+3$, by Lemma \ref{lemma7} we have that $\underbrace{f \oplus f \cdots \oplus f}_{\text{$t$-times}} \oplus  \ \varphi$ is an antiautomorphism of $\mathbf{Z}_{2}^{r}$. This completes the proof. 
\end{proof}

\begin{remark} If one prefers to avoid coding, we now present a linear algebra approach to provide examples of (some) antiautomorphisms of certain finite elementary $p$-abelian groups. Let $\alpha \geq 2$ and suppose we would like to give an example of an antiautomorphism of $\mathbf{Z}_{p}^{\alpha}$. Recall that the set of all bijective $\mathbf{Z}_{p}$-linear maps $\mathbf{Z}_{p}^{\alpha} \rightarrow \mathbf{Z}_{p}^{\alpha}$ can be the identified with the general linear group of matrices $\operatorname{GL}_{\alpha}(\mathbf{Z}_{p})$. Since the set of antiautomorphisms of $\mathbf{Z}_{p}^{\alpha}$ contains $\operatorname{GL}_{\alpha}(\mathbf{Z}_{p})$, it suffices to find an invertible matrix $A \in \mathbf{Z}_{p}^{\alpha \times \alpha}$ such that $1$ is not an eigenvalue of $A$. For instance, if $p=2$ and $\alpha=3$ one can take the following matrix: 
\[A=
\begin{bmatrix}
1 &  1 & 0 \\
0 & 1 & 1 \\
1 & 1 & 1
\end{bmatrix}.
\]
Note that the characteristic polynomial of $A$ is equal to $f(t)=1+t^2+t^3$ and $f(1) \neq 0$ in $\mathbf{Z}_{2}$. It follows that the linear map $h: \mathbf{Z}_{2}^{3} \rightarrow \mathbf{Z}_{2}^{3}$ defined by $h([x],[y],[z])=([x+y],[y+z],[x+y+z])$ is an antiautomorphism of $\mathbf{Z}_{2}^{3}$. Observe that none of the maps given in Proposition \ref{prop9} are linear (because they do not fix the identity); hence not every antiautomorphism is necessarily linear.
\end{remark}
We also have the following result.
\begin{proposition} \label{prop10} For every $m,n \geq 2$, the group $\mathbf{Z}_{2^{m}}^{n}$ admits antiautomorphisms.
\end{proposition}

\begin{proof} Recall that the reduction map $\mathbf{Z}_{2^{m}} \rightarrow \mathbf{Z}_{2}$ induces a surjection $\operatorname{GL}_{n}(\mathbf{Z}_{2^{m}}) \twoheadrightarrow \operatorname{GL}_{n}(\mathbf{Z}_{2})$. For every $n \geq 2$, let $f(t) \in \mathbf{Z}_{2}[t]$ be an irreducible polynomial of degree $n$, and let $C(f)$ be its corresponding Frobenius companion matrix of order $n$. Since $C(f) \in \operatorname{GL}_{n}(\mathbf{Z}_{2})$, then $C(f) \in \operatorname{GL}_{n}(\mathbf{Z}_{2^{m}})$ as well. The result follows.
\end{proof}

\section{Biantiautomorphisms}
\begin{definition}[biantiautomorphism] Let $G$ be a finite abelian group and let $f$ be an antiautomorphism of $G$. We say that $f$ is a biantiautomorphism of $G$ if $f$ is also a linear map.
\end{definition}

The next Proposition gives an example of a group that admits an antiautomorphism but not a biantiautomorphism.
\begin{proposition} \label{prop11} The group $\mathbf{Z}_{2} \oplus \mathbf{Z}_{4}$ admits an antiautomorphism but no biantiautomorphism. 
\end{proposition}
\begin{proof} We first show that $\mathbf{Z}_{2} \oplus \mathbf{Z}_{4}$ admits an antiautomorphism. Indeed, consider the map:
\begin{align*}
(1,3) \mapsto (0,0) \\
(1,2) \mapsto (0,1) \\
(0,3) \mapsto (0,2) \\
(0,2) \mapsto (0,3) \\
(1,0) \mapsto (1,0) \\
(0,1) \mapsto (1,1) \\
(0,0) \mapsto (1,2) \\
(1,1) \mapsto (1,3)
\end{align*}
which is clearly not linear since it does not fixes $(0,0)$. Let us show that $\mathbf{Z}_{2} \oplus \mathbf{Z}_{4}$ does not admit a biantiautomorphism. By \cite[Lemma 11.1]{3} we have that $\operatorname{Aut}_{\mathbf{Z}}(\mathbf{Z}_{2} \oplus \mathbf{Z}_{4}) \cong D_{4}$, the dihedral group on four letters. Now realize $D_{4}$ as the unitriangular matrix group of degree three over $\mathbf{Z}_{2}$ (a.k.a the Heisenberg group modulo $2$). It is easy to check that every matrix in this group has characteristic polynomial equal to $(t-1)^3$ and thus $1$ is always an eigenvalue. Therefore $\mathbf{Z}_{2} \oplus \mathbf{Z}_{4}$ does not admit a biantiautomorphism, as claimed.
\end{proof}

\begin{remark}
The above Proposition shows that for $2$-groups that contain more than one element of order $2$, it is not always the case that biantiautomorphisms exist (even though in the above case an antiautomorphism exists). 
\end{remark}

\begin{lemma} \label{lem3} Let $p$ be a prime, $\alpha \geq 1$ and let $a \in [2,p^{\alpha}-1]$ be such that $(a,p^{\alpha})=1$. Then a non-identity element $\varphi_{a} \in \operatorname{Aut}(\mathbb{Z}_{p^{\alpha}})$ has no non-trivial fixed point if and only if $(a-1,p^{\alpha})=1$.
\end{lemma}

\begin{proof} First suppose that $\varphi_{a}: \mathbb{Z}_{p^{\alpha}} \rightarrow \mathbb{Z}_{p^{\alpha}}$ is a linear automorphism with $0$ as its only fixed point. We claim that $(a-1,p^{\alpha})=1$. Suppose not, then there exists $i \in [1,\alpha-1]$ such that $p^{i}$ divides $a-1$; hence $a=tp^{i}+1$ for some integer $t$. But then $\varphi_{a}(p^{\alpha-i})=ap^{\alpha-i}=(tp^{i}+1)p^{\alpha-i}=tp^{\alpha}+p^{\alpha-i}=p^{\alpha-i}$ in $\mathbb{Z}_{p^{\alpha}}$. It follows that $\varphi_{a}$ has a non-trivial fixed point, a contradiction. Therefore $(a-1,p^{\alpha})=1$. 

On the other hand, suppose that $(a-1,p^{\alpha})=1$ and that there exists $[x]$ such that $\varphi_{a}([x])=[x]$. It follows that $p^{\alpha}$ divides $x(a-1)$. Since $(a-1,p^{\alpha})=1$, then $p^{\alpha}$ divides $x$. It follows that $[x]=[0]$ and $\varphi_{a}$ has no non-trivial fixed point, as claimed.
\end{proof}

\begin{proposition} \label{biantis} Let $\alpha \geq 2$ and let $p$ be an odd prime. Then $\mathbf{Z}_{p^{\alpha}}$ contains exactly $p^{\alpha}-2p^{\alpha-1}$ biantiautomorphisms.
\end{proposition}

\begin{proof} The non-identity elements of $\operatorname{Aut}_{\mathbf{Z}}(\mathbf{Z}_{p^{\alpha}})$ are given by the multiplication maps $\varphi_{a}([t])=[at]$ where $[a] \in (\mathbf{Z}_{p^{\alpha}})^{\ast} \setminus \{[1]\}$. Therefore, it suffices to count the number of maps $\varphi_{a}$ which have no non-trivial fixed point. By Lemma \ref{lem3}, this happens if and only if $(a-1,p^{\alpha})=1$; and by the proof of Proposition \ref{prop4}, this number is exactly $p^{\alpha}-2p^{\alpha-1}$. This completes the proof.
\end{proof}

\begin{theorem} Let $G$ be a cyclic group of odd order $n$. Then the number of biantiautomorphisms is given by $\displaystyle \prod_{p \in \mathbf{P}_{n}} p^{\alpha-1}(p-2)$, where $\mathbf{P}_{n}$ is the set of prime divisors of $n$ and $\alpha$ is the greatest integer such that $p^{\alpha}$ divides $n$.
\end{theorem}

\begin{proof} Let $n=\displaystyle \prod_{i=1}^{k} p_{i}^{\alpha_{i}}$ where the $p_{i}$ are distinct odd primes and $\alpha_{i} \geq 1$. By the Chinese Remainder Theorem, there exists an isomorphism of abelian groups $G \cong \displaystyle \bigoplus_{i=1}^{k} \mathbb{Z}_{p_{i}^{\alpha_{i}}}$. As the orders of the direct summands of $G$ are pairwise relatively prime, then $\operatorname{Aut}_{\mathbf{Z}}(G) \cong \displaystyle \bigoplus_{i=1}^{k} \operatorname{Aut}_{\mathbf{Z}}(\mathbf{Z}_{p_{i}^{\alpha_{i}}})$. Therefore the biantiautomorphisms of $G$ are completely determined by the number of biantiautomorphisms of its direct summands. By Proposition \ref{biantis}, for each $i \in \{1,\ldots,k\}$ there are exactly $p_{i}^{\alpha_{i}}-2p_{i}^{\alpha_{i}-1}=p_{i}^{\alpha_{i}-1}(p_{i}-2)$ biantiautomorphisms. Taking the product over all prime divisors of $n$ yields the desired result.
\end{proof}

\begin{definition}[fixed point free automorphism] Let $f$ be an automorphism of a finite abelian group $G$. We say that $f$ is a fixed point free automorphism of $G$ if $f(x) \neq x$ for all $x \in G \setminus \{0\}$.
\end{definition}

The following statement asserts that certain $2$-groups which contain more than one element of order $2$, never admit biantiautomorphisms of prime order. However, this does not rule out the existence of antiautomorphisms or biantiautomorphisms of non-prime order.

\begin{proposition} \label{prop12} Let $d_{1}, \hdots, d_{m}$ be pairwise distinct integers and let $G=\displaystyle \bigoplus_{i=1}^{m} \mathbf{Z}_{2^{d_{i}}}$. Then $G$ does not admit a biantiautomorphism of prime order.
\end{proposition}

\begin{proof} Suppose that $f: G \rightarrow G$ is a biantiautomorphism of prime order. Since $f$ is an antiautomorphism, then $f$ is fixed point free. Now note that if $f: G \rightarrow G$ is a fixed point free automorphism of prime order, then the cyclic subgroup $\langle f \rangle \subseteq \operatorname{Aut}(G)$ is a fixed point free automorphism group (meaning every non-identity element of it is fixed point free). By \cite[Corollary 6.10]{4} we have that $G$ does not have a fixed point free automorphism group; consequently, no fixed point free automorphism of prime order exists. This completes the proof.
\end{proof}
Using Propositions \ref{prop5}, \ref{prop6}, \ref{prop9} and \ref{prop10}, we obtain a partial classification of the finite abelian groups which admit antiautomorphisms. We summarize this classification in the following
\begin{theorem} Let $G$ be a finite abelian group. Then:
\begin{enumerate}
\item If $G$ has no elements of order $2$, then $G$ admits an antiautomorphism.
\item If $G$ has exactly one element of order $2$, then $G$ does not admit antiautomorphisms.
\item If $G=\mathbf{Z}_{2^{m}}^{n}$, where $n \geq 2$ and $m \geq 1$, then $G$ admits an antiautomorphism.
\end{enumerate}
\end{theorem}
Finally, we end the paper with the following questions: \\

\textbf{Question $1$}. Is there an exact formula for the number of antiautomorphisms of $\mathbf{Z}_{p}^{n}$ for any prime number $p$ and $n \geq 2$? \\

In \cite{6}, C. Ryan gives a recursive relation for computing the exact number of biantiautomorphisms of $\mathbf{Z}_{p}^{n}$ for any prime number $p$ and $n \geq 2$. Note that this problem is equivalent to finding the number of elements of $\operatorname{GL}_{n}(\mathbf{Z}_{p})$ which have no non-trivial fixed point. However, the general case is not covered since an antiautomorphism is not necessarily linear. \\
 
\textbf{Question $2$}. Which finite abelian $2$-groups admit antiautomorphisms but not biantiautomorphisms? \\

\textbf{Acknowledgments.} We thank Pete L. Clark, Florian Luca and Efr\'{e}n Per\'{e}z Terrazas for very useful remarks and helpful comments on a first draft of this paper. We also thank Peter Mayr for clarifying some results of his thesis.

\end{document}